\documentclass[12pt,reqno]{amsart}
\usepackage{amsmath,amsthm,amscd,amsfonts,amssymb,color}
\usepackage{cite}
\usepackage[mathscr]{eucal}
\usepackage{mathtools}

\usepackage[bookmarksnumbered,colorlinks,plainpages]{hyperref}
\newtheorem{theorem}{Theorem}[section]
\newtheorem{lemma}[theorem]{Lemma}

\theoremstyle{definition}
\newtheorem{definition}[theorem]{Definition}

\newtheorem{Open Prob}[theorem]{Open Problem}
\theoremstyle{remark}

\numberwithin{equation}{section}
\def\DJ{\leavevmode\setbox0=\hbox{D}\kern0pt\rlap{\kern.04em\raise.188\ht0\hbox{-}}D}
\begin{document}

\title[$F$-Proximal contraction]{A note on the paper "Best proximity point of generalized $F$-proximal non-self contractions}

\author[S.\ Som]
{Sumit Som}

\address{           Sumit Som,
                    Department of Mathematics,
                    School of Basic and Applied Sciences, Adamas University, Barasat-700126, India.}
                    \email{somkakdwip@gmail.com}



\begin{abstract}
In the year 2021, Beg et al. \cite{beg} [J. Fixed Point Theory Appl.(2021)] introduced two classes of non-self mappings namely, generalized $F$-proximal contraction of the first kind and generalized $F$-proximal contraction of the second kind. Then authors studied the existence and uniqueness of best proximity points for this two classes of mappings. In this short note, we show that the existence of best proximity point for generalized $F$-proximal contraction of the first kind follows from the same conclusion in fixed point theory.
\end{abstract}

\maketitle

\section{\bf{Introduction}}

Let $(X,d)$ be a metric space and $f:M\rightarrow N$ be a mapping where $M,N$ are non-empty subsets of the metric space $X.$ If $f(M)\cap M \neq \phi$ then one search for a necessary and sufficient condition under which the mapping $f$ has a fixed point. One of the fundamental results in metric fixed point theory is the Banach contraction principle. In the year 1922, Banach \cite{BA} proved that if $X$ is a complete metric space and $f:X\rightarrow X$ is a contraction mapping then the mapping $f$ has a unique fixed point. Banach contraction principle has a lot of applications in differential equations, integral equations for the existence of solutions. But for a mapping $f:M\rightarrow N$ if $f(M)\cap M =\phi$ then the mapping $f$ has no fixed points. In this case, one interesting problem is to search for an element $x\in M$ such that $d(x,f(x))= d(M,N)$ where $d(M,N)=\inf~\{d(x,y):x\in M, y\in N\}.$ Best proximity point problems deal with this situation. Authors often prove best proximity point results to generalize fixed point results for self mappings. To generalize the notion of Banach contraction principle, In the year 2012, Wardowski \cite{war} introduced the notion of $F$-contraction mapping and investigated about unique fixed point of such class of mappings in a complete metric space. Then in the year 2014, Cosentino and Vetro \cite{cosen} introduced the notion of $F$-contraction of Hardy-Rogers-type as a generalization of $F$-contraction mapping and showed the existence of fixed points for such class of mappings in a complete metric space. Then in the year 2021, Beg et al. \cite{beg} introduced the notion of generalized $F$-proximal contraction of the first kind and generalized $F$-proximal contraction of the second kind respectively for non-self mappings and investigated about unique best proximity point for these two classes of non-self mappings to generalize the fixed point result of Cosentino and Vetro \cite{cosen}. In this short note, we show that the existence of best proximity point for generalized $F$-proximal contraction of the first kind follows from the same conclusion in fixed point theory.

\section{\bf{Main results}}
Throughout this article $\mathbb{R}^{+}$ denotes the set of all positive real numbers and $\mathbb{R}$ denotes the set of all real numbers. Let $f:\mathbb{R}^{+}\rightarrow \mathbb{R}$ be a mapping satisfying the following conditions:

i) $f$ is strictly increasing;\\
ii) for every sequence $\{\alpha_n\}\subset \mathbb{R}^{+}$ we have $\displaystyle{\lim_{n\rightarrow \infty}}\alpha_n=0\Longleftrightarrow \displaystyle{\lim_{n\rightarrow \infty}}f(\alpha_n)=-\infty;$\\
iii)there exists $k\in (0,1)$ such that $\displaystyle{\lim_{\alpha\rightarrow 0+}}\alpha^{k}f(\alpha)=0.$

In this paper the set of all functions $f:\mathbb{R}^{+}\rightarrow \mathbb{R}$ satisfying the three conditions above will be denoted by $\Omega.$

We first recall the definition of $F$-contraction mapping from \cite{war} as follows.

\begin{definition}\cite{war}
Let $(X,d)$ be a metric space and $F\in \Omega.$ A mapping $T:X\rightarrow X$ is said to be an $F$-contraction mapping if there exists $\tau>0$ such that
$$\forall x,y\in X, d(Tx,Ty)>0\Rightarrow \tau+F(d(Tx,Ty))\leq F(d(x,y)).$$ 
\end{definition}

Now we recall the definition of $F$-contraction of Hardy-Rogers-type from \cite{cosen} as follows.

\begin{definition}\cite{cosen} \label{a1}
Let $(X,d)$ be a metric space. A mapping $T:X\rightarrow X$ is said to be an $F$-contraction of Hardy-Rogers-type if there exists $F\in \Omega$, $a,b,c,e,\tau>0,$ $L\geq 0$ with $a+b+c+2e=1, c\neq 1$ such that 
$$\forall x,y\in X, d(Tx,Ty)>0\Rightarrow \tau+F(d(Tx,Ty))\leq $$
$$F\Big(ad(x,y)+bd(x,Tx)+cd(y,Ty)+ed(x,Ty)+Ld(y,Tx)\Big).$$ In this paper, we will say $T$ is an $F$-contraction mapping with coefficients $a,b,c,e,\tau,L.$
\end{definition}
By using this definition \ref{a1}, Cosentino and Vetro \cite{cosen} proved the following fixed point theorem.

\begin{theorem}\cite{cosen} \label{a2}
Let $(X,d)$ be a complete metric space and $T:X\rightarrow X$ be a self mapping such that $T$ is an $F$-contraction mapping with coefficients $a,b,c,e,\tau,L.$ Then $T$ has a fixed point. Moreover, if $a+e+L\leq 1$ then $T$ has a unique fixed point.
\end{theorem}

Now, we recall the definition of generalized $F$-proximal contraction of the first kind from \cite{beg} as follows.

\begin{definition}\cite{war}\label{a3}
Let $(A, B)$ be a pair of non-empty subsets of a metric space $(X,d).$ A non-self mapping $f:A\rightarrow B$ is said to be a generalized $F$-proximal contraction of the first kind if there exists $F\in \Omega$, $a,b,c,h,\tau>0$ with $a+b+c+2h=1, c\neq 1$ such that $d(u_1, f(x_1))= d(A, B)$ and $d(u_2, f(x_2))= d(A, B)$ will imply
$$\tau+F(d(u_1,u_2))\leq $$
$$F\Big(ad(x_1,x_2)+bd(u_1,x_1)+cd(u_2,x_2)+h(d(u_2,x_1)+d(u_1,x_2))\Big)$$
for all $u_1, u_2, x_1, x_2 \in A$ and $u_1\neq u_2.$ 
\end{definition}

The following notations will be needed.
Let $(X,d)$ be a metric space and $A,B$ be nonempty subsets of $X.$ Then
$$A_0=\{x\in A: d(x,y)=d(A,B)~\mbox{for some}~y\in B\}.$$
$$B_0=\{y\in B: d(x,y)=d(A,B)~\mbox{for some}~x\in A\}.$$

\begin{definition}\cite{beg}
Let $(X,d)$ be a metric space and $A,B$ be two non-empty subsets of $X.$ Then $B$ is said to be approximatively compact with respect to $A$ if for every sequence $\{y_n\}$ of $B$ satisfying $d(x,y_n)\rightarrow d(x,B)$ as $n\rightarrow \infty$ for some $x\in A$ has a convergent subsequence where $d(x,B)=\inf~\{d(x,y): y\in B\}.$
\end{definition}

\begin{definition}\cite{beg}
Let $(X,d)$ be a metric space and $A,B$ be two non-empty subsets of $X$ with $A_0\neq \phi.$ If for every $u_1,u_2\in A$ and $v_1,v_2\in B$ \[
\begin{rcases}
 d(u_1, v_1)= d(A, B)\\
 d(u_2, v_2)= d(A, B)
 \end{rcases}
  {\Longrightarrow d(u_1,u_2)=d(v_1,v_2)}
  \] 
then the pair $(A,B)$ is said to have the $p$-property.
\end{definition}

We need the following result from \cite{FE}.
\begin{lemma}\cite[\, proposition 3.3]{FE}\label{b}
Let $(A,B)$ be a nonempty and closed pair of subsets of a metric space $(X,d)$ such that $B$ is approximatively compact with respect to $A.$ Then $A_0$ is closed.
\end{lemma}

By using definition \ref{a3}, Beg et al. proved the following best proximity point result in \cite{beg}

\begin{theorem}\cite{beg}\label{a4}
Let $A,B$ be non-void, closed subsets of a complete metric space $(X,d)$ such that $B$ is approximatively compact with respect to $A.$ Let $T:A\rightarrow B$ satisfies the following conditions:\\
i) $T(A_0)\subseteq B_0$ and $(A,B)$ satisfies the $p$-property;\\
ii) $T$ is a generalized $F$-proximal contraction of the first kind.\\
Then there exists a unique element $x\in A$ such that $d(x,T(x))=d(A,B).$ Further, for any $x_0\in A_0,$ the sequence $\{x_n\}$ defined by $d(x_{n+1},T(x_n))=d(A,B),$ converges to the best proximity point $x.$
\end{theorem}


Now we like to state our main result.
\begin{theorem}\label{a5}
Theorem \ref{a4} follows from Theorem \ref{a2} and the Picard sequence will work to find the unique best proximity point $x\in A$ satisfying $d(x, T(x))=d(A,B).$
\end{theorem}

\begin{proof}
Let $x\in A_0.$ As $T(A_0)\subseteq B_0,$ so, $T(x)\in B_0.$ So, there exists $y\in A_0$ such that $d(y,T(x))=d(A,B).$ Now, we will show that $y\in A_0$ is unique. Suppose there exists $y_1,y_2\in A_0$ such that $d(y_1,T(x))=d(A,B)$ and $d(y_2,T(x))=d(A,B).$ Now since the pair $(A,B)$ have the $p$-property so, we have $$d(y_1,y_2)=d(Tx,Tx)$$
$$\Rightarrow d(y_1,y_2)=0$$
$$\Rightarrow y_1=y_2.$$
Now define a mapping $S:A_0\rightarrow A_0$ by $Sx=y$ having the property that $d(Sx,Tx)=d(A,B).$ Now we show that $S$ is a $F$-contraction of Hardy-Rogers-type. Let $x_1,x_2\in A_0$ with $S(x_1)\neq S(x_2).$ Now since $d(S(x_1),T(x_1))=d(A,B)$ and $d(S(x_2),T(x_2))=d(A,B)$ and $T$ is a generalized $F$-proximal contraction of the first kind, so, we have,
$$\tau+F(d(S(x_1),S(x_2)))\leq $$
$$F\Big(ad(x_1,x_2)+bd(S(x_1),x_1)+cd(S(x_2),x_2)+h(d(S(x_2),x_1)+d(S(x_1),x_2))\Big)$$
$$\Longrightarrow \tau+F(d(S(x_1),S(x_2)))\leq $$
$$F\Big(ad(x_1,x_2)+bd(S(x_1),x_1)+cd(S(x_2),x_2)+hd(S(x_2),x_1)+hd(S(x_1),x_2)\Big)$$ where $a+b+c+2h=1, c\neq 1$ and $h>0.$ This shows that the mapping $S:A_0\rightarrow A_0$ is a $F$-contraction of Hardy-Rogers-type. Also, from lemma \ref{b}, we can conclude that $A_0$ is closed and hence is a complete metric space. So, from Theorem \ref{a2}, we can say, the mapping $S$ has a fixed point in $A_0.$ So, there exists $z\in A_0$ such that $S(z)=z.$ Also, $d(z,T(z))=d(Sz,Tz)=d(A,B).$ This shows that $z$ is a best proximity point for the mapping $T:A\rightarrow B$ and the picard sequence $\{S^{n}(x_0)\}$ will converge to the best proximity point for the mapping $T:A\rightarrow B$ for any initial point $x_0\in A_0$ as follows from the proof of Theorem \ref{a2} in \cite{cosen} by Cosentino and Vetro. The proof for the uniqueness of the best proximity point for the generalized $F$-proximal contraction of the first kind mapping $T$ is same as Beg et al. already proved in \cite{beg}, so omitted.
\end{proof}

\section{conclusion}
The main motivation of the current paper is to show that the best proximity point result proved by Beg et al. in \cite{beg} follows from the fixed point result of Cosentino and Vetro in \cite{cosen}. Also, we show that the picard sequence will also work along with the algorithm proposed by Beg et al. in \cite{beg} to find the unique best proximity point of the generalized $F$-proximal contraction of the first kind mapping $T:A\rightarrow B.$

\end{document}